\documentclass[12pt]{amsart}

\usepackage{color}
\usepackage{amsmath, amsthm, amssymb}
\usepackage{amsfonts}
\usepackage[ansinew]{inputenc}
\usepackage[dvips]{epsfig}
\usepackage{graphicx}
\usepackage[english]{babel}
\usepackage{hyperref}
\theoremstyle{plain}
\newtheorem{thm}{Theorem}[section]
\newtheorem{cor}[thm]{Corollary}
\newtheorem{lem}[thm]{Lemma}
\newtheorem{prop}[thm]{Proposition}

\theoremstyle{definition}

\theoremstyle{remark}
\newtheorem{rem}[thm]{Remark}

\numberwithin{equation}{section}

\newcommand{\average}{{\mathchoice {\kern1ex\vcenter{\hrule height.4pt
width 6pt depth0pt} \kern-9.7pt} {\kern1ex\vcenter{\hrule
height.4pt width 4.3pt depth0pt} \kern-7pt} {} {} }}

\def\R{\mathbb{R}}

\def\div{\text{div}}

\begin{document}

\title{Sobolev and isoperimetric inequalities with monomial weights}

\author{Xavier Cabr\'e}

\address{ICREA and Universitat Polit\`ecnica de Catalunya,
Departament de Matem\`{a}tica  Aplicada I, Diagonal 647, 08028 Barcelona, Spain}

\email{xavier.cabre@upc.edu}


\author{Xavier Ros-Oton}

\address{Universitat Polit\`ecnica de Catalunya, Departament de Matem\`{a}tica  Aplicada I, Diagonal 647, 08028 Barcelona, Spain}
\email{xavier.ros.oton@upc.edu}


\begin{abstract} We consider the monomial weight $|x_1|^{A_1}\cdots|x_n|^{A_n}$ in $\R^n$, where $A_i\geq0$ is a real number for each $i=1,...,n$,
and establish Sobolev, isoperimetric, Morrey, and Trudinger inequalities involving this weight.
They are the analogue of the classical ones with the Lebesgue measure $dx$ replaced by $|x_1|^{A_1}\cdots|x_n|^{A_n}dx$, and they contain the best or critical exponent (which depends on $A_1$, ..., $A_n$).
More importantly, for the Sobolev and isoperimetric inequalities, we obtain the best constant and extremal functions.

When $A_i$ are nonnegative \textit{integers}, these inequalities are exactly the classical ones in the Euclidean space $\R^D$
(with no weight) when written for axially symmetric functions and domains
in $\R^D=\R^{A_1+1}\times\cdots\times\R^{A_n+1}$.
\end{abstract}

\maketitle

\section{Introduction and results} \label{intro}

In this paper we establish Sobolev, Morrey, Trudinger, and isoperimetric inequalities in $\mathbb R^n$ with the weight $x^A$,
where $A=(A_1,...,A_n)$ and
\begin{equation}\label{monomialweight}x^A:=|x_1|^{A_1}\cdots|x_n|^{A_n},\qquad A_1\geq0,\ ...,\ A_n\geq0.
\end{equation}
They were announced in our previous article \cite{CR}. In fact,
their interest and applications arose in \cite{CR}, where we had $n=2$ in \eqref{monomialweight}.
In that paper we studied the regularity of stable solutions to reaction-diffusion problems
in bounded domains of double revolution in $\R^N$.
That is, domains of $\R^N$ which are invariant under rotations of the first $m$ variables and of the last $N-m$ variables, i.e.,
\[\overline\Omega=\{(x^1,x^2)\in\mathbb R^m\times\mathbb R^{N-m}:(s=|x^1|,t=|x^2|)\in\overline\Omega_2\},\]
where $\Omega_2\subset(\mathbb R_+)^2$ is a bounded domain.

The first step towards the results in \cite{CR} consisted of obtaining bounds for some integrals of the form
\[\int_{\Omega_2} \left\{s^{-\alpha} u_s^2+t^{-\beta}u_t^2\right\}ds\, dt,\]
where $u$ is any stable solution and $s$ and $t$ are, as above, the two radial coordinates describing $\Omega$.
Then, from this bound we needed to deduce that $u\in L^q(\Omega)$, with $q$ as large as possible.
After a change of variables of the form $s=\sigma^{\gamma_1}$, $t=\tau^{\gamma_2}$, what we needed to establish is the following Sobolev inequality.
Given $a>-1$ and $b>-1$, find the greatest exponent $q$ for which
\[\left(\int_{\widetilde\Omega_2}\sigma^a\tau^b|u|^{q}d\sigma d\tau\right)^{1/q}\leq C\left(\int_{\widetilde\Omega_2}\sigma^a\tau^b|\nabla u|^2d\sigma d\tau\right)^{1/2}\]
holds for all smooth functions $u$ vanishing on $\partial \widetilde\Omega_2\cap(\R_+)^2$, where $\widetilde\Omega_2=\{(\sigma,\tau)\in (\R_+)^2:(s=\sigma^{\gamma_1},t=\tau^{\gamma_2})\in\Omega_2\}$ is an arbitrary bounded domain of $(\R_+)^2$.

On the one hand, we obtained that $u\in L^\infty(\widetilde \Omega_2)$ whenever the right hand side is finite for some $a$, $b$ with $a+b<0$.
On the other hand, in case $a+b>0$ we established the following.

Throughout the paper, $C^1_c(\R^n)$ denotes the space of $C^1$ functions with compact support in $\R^n$.

\begin{prop}[\cite{CR}] Let $a$ and $b$ be real numbers such that
\[a>-1,\ \ b>-1,\ \ \mbox{and}\ \ a+b>0.\]
Let $u$ be a nonnegative $C^1_c(\R^2)$ function such that
\begin{equation}\label{assumption}
u_\sigma\leq0\ \ \textrm{and}\ \ u_\tau\leq0\ \ \textrm{in}\ \ \{\sigma>0, \tau>0\}.
\end{equation}
with strict inequalities in the set $\{u>0\}$.
Then, there exists a constant $C$, depending only on $a$ and $b$, such that
\begin{equation}\label{sobsigmatau}
\left(\int_{\{\sigma>0,\,\tau>0\}}\sigma^a\tau^b|u|^{2_*}d\sigma d\tau\right)^{1/2_*}\leq C\left(\int_{\{\sigma>0,\, \tau>0\}}\sigma^a\tau^b|\nabla u|^2d\sigma d\tau\right)^{1/2},
\end{equation}
where $2_*=\frac{2D}{D-2}$ and $D=a+b+2$.
\end{prop}

In \cite{CR} we also obtained Sobolev inequalities with other powers $|\nabla u|^p$, $1\leq p<D$.
By a standard scaling argument one sees that the exponent $2_*=\frac{2D}{D-2}$ in \eqref{sobsigmatau} is optimal, in the sense that \eqref{sobsigmatau} can not hold with any other exponent larger than this one.
In addition, when $a<0$ or $b<0$ inequality \eqref{sobsigmatau} is not valid without assumption \eqref{assumption}; see Remark \ref{remass} for more details.

\begin{rem}\label{rementers}
When $a$ and $b$ are positive \textit{integers}, inequality \eqref{sobsigmatau} is exactly the classical Sobolev inequality in $\R^D=\mathbb R^{a+1}\times\mathbb R^{b+1}$ for functions which are radially symmetric on the first $a+1$ variables and on the last $b+1$ variables.

Indeed, for each $z\in \mathbb R^D$ write $z=(z^1,z^2)$, with $z^1\in\mathbb R^{a+1}$ and $z^2\in\mathbb R^{b+1}$, and define $(\sigma,\tau)=(|z^1|,|z^2|)\in \{\sigma\geq0,\tau\geq0\}$.
Now, for each function $u$ in $(\R_+)^2$ we define $\tilde u(z)=u(|z_1|,|z_2|)$.
We have that $|\nabla_z \tilde u|=|\nabla_{(\sigma,\tau)} u|$.
Moreover, an integral over $\mathbb R^D$ of a function depending only on $|z^1|$ and $|z^2|$ can be written as an integral in $(\mathbb R_+)^2$ with $dz=c_{a,b}\sigma^a\tau^b d\sigma d\tau$ for some constant $c_{a,b}$. Therefore, writing in the coordinates $(\sigma,\tau)$ the classical Sobolev inequality in $\mathbb R^D$ for the function $\tilde u$, we obtain the validity of \eqref{sobsigmatau}.
Note that if $a>0$ and $b=0$ then we obtain the inequality in $\{\sigma>0\}$ instead of $\{\sigma>0,\tau>0\}$, that is, $\left(\int_{\{\sigma>0\}}\sigma^a|u|^{2_*}d\sigma d\tau\right)^{1/2_*}\leq C\left(\int_{\{\sigma>0\}}\sigma^a|\nabla u|^2d\sigma d\tau\right)^{1/2}$ --- and this motivates definition \eqref{R*} below in the case of a general monomial $x^A$.
\end{rem}

The same argument as in the previous remark, but now with multiple axial symmetries, shows the following.
When $A_1,...,A_n$ are nonnegative \textit{integers}, the Sobolev, isoperimetric, Morrey, and Trudinger inequalities with the monomial weight
\[x^A=|x_1|^{A_1}\cdots|x_n|^{A_n}\]
are exactly the classical ones in
\[\R^{A_1+1}\times\cdots\times\R^{A_n+1}\]
when written in radial coordinates for functions which are radially symmetric with respect to the first $A_1+1$ variables, also with respect to the next $A_2+1$ variables, and so on until radial symmetry with respect to the last $A_n+1$ variables.

The aim of this paper is to extend inequality \eqref{sobsigmatau} in $\R^2$ to the case of $\R^n$ with any weight of the form \eqref{monomialweight}, i.e., of the form $x^A=|x_1|^{A_1}\cdots|x_n|^{A_n}$.
When $A_i$ are nonnegative \textit{real} numbers, we prove that this weighted Sobolev inequality holds for any function $u\in C^1_c(\R^n)$ --- and thus assumption \eqref{assumption} is not necessary.
We obtain also Sobolev inequalities with $|\nabla u|^2$ replaced by other powers $|\nabla u|^p$.
More importantly, we find the best constant and extremal functions in these inequalities.
For this, a crucial ingredient is a new isoperimetric inequality  involving the weight $x^A$ and with best constant.
This is Theorem \ref{isoperimetric} below, a main result of this paper.
In addition, we prove Morrey and Trudinger type inequalities involving the monomial weight.
All these results were announced in our previous paper \cite{CR}.

The first result of the paper is the Sobolev inequality with a monomial weight, and reads as follows.
Here, and in the rest of the paper, we denote
\begin{equation}\label{R*} \mathbb R_*^n=\{(x_1,...,x_n)\in \mathbb R^n: x_i>0\ \textrm{whenever}\ A_i>0\}\end{equation}
and
\[B_r^*=B_r(0)\cap \R^n_*.\]
For each $1\leq p<\infty$, let $W^{1,p}_0(\R^n,x^Adx)$ be the closure of the space of $C^1_c(\R^n)$ under the norm $\left(\int_{\R^n}x^A(|u|^p+|\nabla u|^p)dx\right)^{1/p}$.

\begin{thm}\label{sob} Let $A$ be a nonnegative vector in $\mathbb R^n$, $D=A_1+\cdots+A_n+n$, and $1\leq p< D$ be a real number. Then,
\begin{itemize}
\item[(a)] There exists a constant $C_p$ such that for all $u\in C^{1}_c(\mathbb R^n)$,
    \begin{equation}\label{sobolev}\left(\int_{\mathbb R_*^n}x^A|u|^{p_*}dx\right)^{\frac{1}{p_*}}\leq C_p\left(\int_{\mathbb R_*^n}x^A|\nabla u|^pdx\right)^{\frac1p},\end{equation}
    where $p_*=\frac{pD}{D-p}$ and $x^A$ is given by \eqref{monomialweight}.
\item[(b)] The best constant $C_p$ is given by the explicit expression \eqref{bestcnt}-\eqref{bestcnt2}.
    When $p=1$, this constant is not attained in $W^{1,1}_0(\R^n,x^Adx)$.
    Instead, when $1<p<D$ it is attained in $W^{1,p}_0(\R^n,x^Adx)$ by the functions
    \begin{equation}\label{extr} u_{a,b}(x)=\left(a+b|x|^{\frac{p}{p-1}}\right)^{1-\frac{D}{p}},\end{equation}
    where $a$ and $b$ are any positive constants.
\end{itemize}
\end{thm}

Note that the exponent $p_*$ is exactly the same as in the classical Sobolev inequality, but in this case the ``dimension'' is given by
$D$ instead of $n$. Note also that when $A_1=...=A_n=0$ then $D=n$ and \eqref{sobolev} is exactly the classical Sobolev inequality.
As before, a scaling argument shows that the exponent $p_*$ is optimal, in the sense that \eqref{sobolev} can not hold with any other exponent.

Note that the integrals in \eqref{sobolev} are computed over $\R^n_*$ but the functions $u$ involved need not vanish on the coordinate hyperplanes on $\partial\R^n_*$.
Let us mention that $u_{a,b}$ are extremal functions for inequality \eqref{sobolev}, but we do not know if these are all extremal functions for the inequality --- except in the case when all $A_i$ are integers.

The Sobolev inequalities in all of $\R^n$ follow easily (without the best constant) from the ones in $\R^n_*$ by applying them at most $2^n$ times (one for each hyperoctant of $\R^n$, that is, for each set $\{\epsilon_ix_i>0,\ i=1,...,n\}$, where $\epsilon_i\in\{-1,1\}$) and adding up the obtained inequalities.
Consider now functions $u\in C^1_c(\R^n)$ that are even with respect to those variables $x_i$ for which $A_i>0$.
They arise naturally in nonlinear problems in $\R^D$ whenever $D$ is an integer (see \cite{CR}).
Among these functions, the Sobolev inequality in all of $\R^n$ has also as extremals the functions $u_{a,b}$ in \eqref{extr}.

After a change of variables of the form $x_i=y_i^{\gamma_i}$, \eqref{sobolev} yields new inequalities of the form
\[\|u\|_{L^{p_*}(\R^n_*)}\leq C\sum_{i=1}^n \|x_i^{\alpha_i}u_{x_i}\|_{L^p(\R^n_*)},\]
where $\alpha_i$ are arbitrary exponents in $[0,1)$; see Corollary \ref{cor}.
In these inequalities, the exponent on the left hand side is given by $p_*=\frac{pD}{D-p}$, where $D=n+\frac{\alpha_1}{1-\alpha_1}+\cdots+\frac{\alpha_n}{1-\alpha_n}$.

When $p>1$ and $A_i<p-1$ for all $i=1,...,n$, the weight \eqref{monomialweight} belongs to the Muckenhoupt class $A_p$, and thus part (a) --- without the best constant and for bounded domains --- can be deduced from some classical results on weighted Sobolev inequalities.
Indeed, it follows from a classical result of Fabes-Kenig-Serapioni \cite{FKS} that for any bounded domain $\Omega\subset \R^n$ there exists $q>p$ for which
$\|u\|_{L^q(\Omega,x^Adx)}\leq C\|u\|_{W^{1,p}_0(\Omega,x^Adx)}$ holds.
Moreover, the optimal exponent $q=p_*$ can be found by using a result of Hajlasz \cite[Theorem 6]{H}.
However, in general the monomial weight \eqref{monomialweight} does not satisfy the Muckenhoupt condition $A_p$ and Theorem \ref{sob} cannot be deduced from these results on weighted Sobolev inequalities, even without the best constant in the inequality.

The main ingredient in the proof of Theorem \ref{sob} is a new weighted isoperimetric inequality with best constant, given by Theorem \ref{isoperimetric} below.
Let us mention that if one is willing not to have the best constant in the Sobolev inequality, we give an alternative and more elementary proof of part (a) of Theorem \ref{sob} under some additional hypotheses.
Namely, we assume $A_i>0$ for all $i$ and $u_{x_i}\leq0$ in $\{x_i>0,\ i=1,...,n\}$ --- an assumption equivalent to \eqref{assumption} in Proposition \ref{sobsigmatau} and which suffices for some applications to nonlinear problems.

The following is the new isoperimetric inequality with a monomial weight.

\begin{thm}\label{isoperimetric} Let $A$ be a nonnegative vector in $\R^n$, $x^A$ given by \eqref{monomialweight}, and $D=A_1+\cdots+A_n+n$.
Let $\Omega\subset \R^n$ be a bounded Lipschitz domain. Denote
\begin{equation*}\label{m(E)} m(\Omega)=\int_\Omega x^Adx\qquad\textrm{and}\qquad P(\Omega)=\int_{\partial\Omega} x^Ad\sigma.\end{equation*}
Then,
\begin{equation} \label{isop}
\frac{P(\Omega)}{m(\Omega)^{\frac{D-1}{D}}} \geq \frac{P(B_1^*)}{m(B_1^*)^{\frac{D-1}{D}}} \ ,
\end{equation}
where $B_1^*=B_1(0)\cap \mathbb R^n_*$ is the unit ball intersected with $\mathbb R^n_*$, and $\R^n_*$ is given by \eqref{R*}.
\end{thm}

It is a surprising fact that the weight $x^A$ is not radially symmetric but still Euclidean balls centered at the origin (intersected with $\R^n_*$) minimize this isoperimetric quotient.

Recently, these type of isoperimetric inequalities with weights (also called ``with densities'') have attracted much attention; see the nice survey of F. Morgan in the Notices of the AMS \cite{M}.
In a forthcoming paper \cite{CRS} we will prove new weighted isoperimetric inequalities in convex cones of $\R^n$ that extend Theorem \ref{isoperimetric};
some of them have been announced in \cite{CRS-CRAS}.

Equality in \eqref{isop} holds when $\Omega=B_r^*=B_r(0)\cap \R^n_*$, where $r$ is any positive number.
We expect these balls centered at the origin intersected with $\R^n_*$ to be the unique minimizers of the isoperimetric quotient.
However, our proof involves the solution of an elliptic equation and due to an issue on its regularity we need to regularize slightly the domain $\Omega$.
This is why we can not obtain that $B_r^*$ are the unique minimizers of \eqref{isop}.
In a future paper \cite{CRS2} (still in progress) we will study the non uniformly elliptic operator \eqref{operator} below and prove some regularity results in $\overline{\R^n_*}$ which may lead to the characterization of equality in the isoperimetric inequality \eqref{isop}.

\begin{rem} Note that, when $A\neq0$, the entire balls $B_r=B_r(0)$ are not minimizers of the isoperimetric quotient.
This is because
\[\frac{P(B_1^*)}{m(B_1^*)^{\frac{D-1}{D}}}=2^{-\frac kD}\frac{P(B_1)}{m(B_1)^{\frac{D-1}{D}}}<\frac{P(B_1)}{m(B_1)^{\frac{D-1}{D}}},\]
where $k$ is the number of positive entries in the vector $A$.
However, if we look for the minimizers of the isoperimetric quotient $P(\Omega)/m(\Omega)^{\frac{D-1}{D}}$ among all sets $\Omega$ which are symmetric with respect to each plane $\{x_i=0\}$ with $i$ such that $A_i>0$, then the balls $B_r(0)$ solve this isoperimetric problem.
\end{rem}

As explained below in Remark \ref{grig}, the fact that $P(\Omega)/m(\Omega)^{\frac{D-1}{D}}\geq c$ for some constant $c>0$ smaller than the one in \eqref{isop} (and hence, nonoptimal) is an interesting consequence of the isoperimetric inequality in product manifolds of A. Grigor'yan \cite{G}.

As said before, our sharp isoperimetric inequality \eqref{isop} is the crucial ingredient needed to prove Theorem \ref{sob} on the Sobolev inequality, especially part (b) on the best constant and on extremals.
Indeed, we prove part (b) by applying our isoperimetric inequality with best constant together with two results of Talenti.
The first one is a radial symmetrization result, which applies since our isoperimetric inequality \eqref{isop} gives the best constant and the sets $B_r(0)\cap \R^n_*$ are extremal sets for any $r>0$.
The second one is a result in dimension 1, which characterizes the minimizers of the functional
\[J(u)=\frac{\left(\int_0^\infty r^{D-1}|u'|^p\right)^{1/p}}{\left(\int_0^\infty r^{D-1}|u|^{p_*}\right)^{1/p_*}},\]
where $p_*=\frac{pD}{D-p}$.

When $n=2$ and $A_1=0$, our Sobolev and isoperimetric inequalities with best constant were already obtained by Maderna and Salsa \cite{MS} in 1981.
Namely, they proved the sharp isoperimetric inequality in $\{(x,y)\in \R^2\,:\, y>0\}$ with weight $y^k$, $k>0$, and from it they deduced the Sobolev inequality with weight $y^k$.
These inequalities arose in the study of an elliptic problem which involved the operator $y^{-k}\textrm{div}(y^k\nabla u)$ in $\{(x,y)\in\mathbb R^2\,:\, y>0\}$, where $k$ is any positive number.
Using symmetrization techniques and their weighted isoperimetric inequality, they obtained sharp estimates for the solution of the problem.
To prove the isoperimetric inequality with weight $y^k$ they first established the existence of a minimizer for the perimeter functional under constraint of fixed area, then computed the first variation of this functional, and finally solved the obtained ODE to deduce that minimizers must be half balls.
Their result can be seen as a particular case of Theorem \ref{isoperimetric} by setting $n=2$ and $A_1=0$.
Our proof of the weighted isoperimetric inequality will be completely different from the one in \cite{MS}, as explained next.

The proof of Theorem \ref{isoperimetric} follows the ideas introduced by the first author in a new proof of the classical isoperimetric inequality;
see \cite{CSCM,CDCDS} or the last edition of Chavel's book \cite{Ch}.
It is quite surprising (and fortunate) that this proof (which gives the best constant)
can be adapted to the case of monomial weights.

The proof of the classical isoperimetric inequality from \cite{CSCM,CDCDS} considers the linear problem
\begin{equation} \label{eqlaplace}
\left\{ \begin{array}{ll} \Delta u = c  &\quad \mbox{in } \Omega\\
\frac{\partial u}{\partial\nu} =1 &\quad \textrm{on }\partial \Omega,
\end{array}\right. \end{equation}
where $c$ is the unique constant for which the problem has a solution.
Then, one uses an argument similar to the Alexandroff-Bakelman-Pucci method (also called ABP method; see for example \cite{GT}) applied to this solution $u$.
Using this argument and the classical inequality between the arithmetic mean (AM) and the geometric mean (GM), the isoperimetric inequality follows.
When $\Omega=B_1$, the solution of \eqref{eqlaplace} is $u(x)=|x|^2/2$ and all inequalities in the proof become equalities.
Here we consider a similar problem to \eqref{eqlaplace} but where the Laplacian is replaced by the operator
\begin{equation}\label{operator}
x^{-A}{\rm div}(x^A\nabla u)=\Delta u+A_1\frac{u_{x_1}}{x_1}+\cdots+A_n\frac{u_{x_n}}{x_n}.\end{equation}
Now, using the same ABP argument with this new problem and a weighted version of the
AM-GM inequality, we obtain \eqref{isop}.
An essential fact in our proof (and this is why $B_1(0)\cap \R^n_*$ is the minimizer) is that the function $u(x)=|x|^2/2$ also solves the equation $x^{-A}{\rm div}(x^A\nabla u)=c$ for some constant $c>0$.
In addition, it has normal derivative $u_\nu=1$ on $\partial B_1$, as in problem \eqref{eqlaplace}.

When $A_1,...,A_n$ are nonnegative integers, the operator
\eqref{operator} is the Laplacian in the space $\R^D=\R^{A_1+1}\times\cdots\times\R^{A_n+1}$ written in radial coordinates.
Thus, if instead $A_i$ are not integers, \eqref{operator} can be seen as some kind of Laplacian in a fractional dimension $D$.
This class of operators was studied by A. Weinstein and others for $n=2$, and the theory on these equations is called ``Generalized Axially Symmetric Potential Theory''; see for example \cite{W}.
In case $A_1=\cdots=A_{n-1}=0$ and $A_n=a\in(-1,1)$, the operator $x^{-A}{\rm div}(x^A\nabla u)$ appears in the re-interpretation of the fractional Laplacian as a local problem in one higher dimension; see \cite{CS}.

The paper \cite{IN} by Ivanov and Nazarov establishes some weighted Sobolev inequalities for $W^{1,p}$ functions with multiple radial symmetries --- a space of functions denoted by $W^{1,p}_{\rm sym}$.
Their result is related to ours in the case in which all the exponents $A_i$ are nonnegative \emph{integers}.
They prove that for functions with multiple radial symmetries in $\R^D$, the embedding $W^{1,p}_{\rm sym}(B_1)\subset L^q(B_1;\,|x|^{\alpha})$, with $p<D$ and $\alpha>0$, holds for some exponents $q$ depending on $\alpha$ that are greater than $p^*=pD/(p-D)$.

Some theorems of trace and interpolation type for functional spaces with weights of the form \eqref{monomialweight} were proved by A. Cavallucci \cite{C} in 1969.
Namely, he established some inequalities of the form
\[\|D^\lambda f\|_{L^p((\R_+)^m\times \{0\},y^Bdy)}\leq C\left(\|f\|_{L^p((\R_+)^n,x^Adx)}+\|D^l f\|_{L^p((\R_+)^n,x^Adx)}\right),\]
where $m\leq n$, $y^B=y_1^{B_1}\cdots y_m^{B_m}$ and $x^A=x_1^{A_1}\cdots x_n^{A_n}$ are two monomial weights, and $\lambda$ and $l$ are multiindices satisfying a certain condition involving $A$, $B$, $m$, $n$, and $p$.
Note that in these inequalities the exponent $p$ is the same in both sides, and thus they are not Sobolev-type inequalities.
To obtain his results, the author used a representation of $D^\lambda f$ in terms of integral transforms of $D^l f$.

The third result of our paper is the weighted version of the Morrey inequality, which reads as follows.

\begin{thm}\label{morrey} Let $A$ be a nonnegative vector in $\mathbb R^n$, $D=A_1+\cdots+A_n+n$, and $p>D$ be a real number.
Then, there exists a constant $C$, depending only on $p$ and $D$, such that
\begin{equation}\label{holder}
\sup_{x\neq y,\ x,\,y\in \R^n_*}\frac{|u(x)-u(y)|}{|x-y|^\alpha}\leq C\left(\int_{\mathbb R^n_*}x^A|\nabla u|^pdx\right)^{1/p}\end{equation}
for all $u\in C^{1}_c(\mathbb R^n)$, where $\alpha=1-\frac{D}{p}$.

As a consequence, if $\Omega\subset\R^n$ is a bounded domain and $u\in C^1_c(\Omega)$ then
\begin{equation}\label{cormorrey}
\sup_{\Omega}|u|\leq C\,{\rm diam}(\Omega)^{1-\frac{D}{p}}\left(\int_\Omega x^A|\nabla u|^pdx\right)^{1/p}.
\end{equation}
\end{thm}

This weighted Morrey inequality will be deduced from the classical one by applying it in dyadic domains and then summing geometric series appropriately; see Section \ref{sec5} for more details.

The next result is the weighted version of the classical Trudinger inequality.

\begin{thm}\label{trudinger} Let $A$ be a nonnegative vector in $\mathbb R^n$, $D=A_1+\cdots+A_n+n$, and $\Omega\subset \R^n$ be a bounded domain.
Then, for each $u\in C^1_c(\Omega)$,
\begin{equation*}\label{trud}
\int_\Omega \exp\left\{\left(\frac{c_1|u|}{\|\nabla u\|_{L^D(\Omega,x^Adx)}}\right)^{\frac{D}{D-1}}\right\} x^Adx\leq C_2m(\Omega),
\end{equation*}
where $m(\Omega)=\int_\Omega x^Adx$, and $c_1$ and $C_2$ are constants depending only on $D$.
\end{thm}

Our proof of this result is based on a bound for the best constant \eqref{bestcnt2} in the weighted Sobolev inequality as $p$ goes to $D$.
Then, the Trudinger inequality will follow by expanding $\exp(\cdot)$ as a power series and applying the weighted Sobolev inequality to each term of the series.
The obtained series is convergent thanks to the mentioned bound for the best constant \eqref{bestcnt2}.

Finally, adding up the results of Theorems \ref{sob}, \ref{morrey}, and \ref{trudinger} we obtain the following continuous embeddings, which are
weighted versions of the classical Sobolev embeddings.

Recall that the Orlicz space $L^\varphi(X,d\mu)$ is defined as the space of measurable functions $u:X\rightarrow\R$ such that
\[\|u\|_{L^\varphi(X,d\mu)}=\inf\left\{K>0: \int_X \varphi\left(\frac{|u|}{K}\right)d\mu\leq1\right\}\]
is finite. Setting $\varphi(t)=t^p$ we recover the definition of the $L^p$ spaces.

\begin{cor}\label{embeddings} Let $A$ be a nonnegative vector in $\mathbb R^n$, $x^A$ be given by \eqref{monomialweight}, and $D=A_1+\cdots+A_n+n$.
Let $k\geq1$ be an integer and $p\geq1$ be a real number.
Then, for any bounded domain $\Omega\subset\R^n$ we have the following continuous embeddings:
\begin{itemize}
\item[(i)] If $kp<D$ then
\[W^{k,p}_0(\Omega,x^Adx)\subset L^{q}(\Omega,x^Adx),\]
where $q$ is given by $\frac{1}{q}=\frac{1}{p}-\frac{k}{D}$.
\item[(ii)] If $kp=D$ then
\[W^{k,p}_0(\Omega,x^Adx)\subset L^\varphi(\Omega,x^Adx),\]
where
\[\varphi(t)=\exp\left(t^{\frac{D}{D-1}}\right)-1.\]
\item[(iii)] If $kp>D$ then
\[W^{k,p}_0(\Omega,x^Adx)\subset C^{r,\alpha}(\overline{\Omega}),\]
where $r=k-[\frac Dp]-1$, and $\alpha=[\frac Dp]+1-\frac Dp$ whenever $\frac Dp$ is not an integer, or $\alpha$ is any positive number smaller than 1 otherwise.
\end{itemize}
\end{cor}

The paper is organized as follows. In section 2 we give the proof of the weighted isoperimetric inequality.
Section 3 establishes the weighted Sobolev inequalities, while in section 4 we obtain their best constants and extremal functions.
Section 5 deals with the weighted Morrey inequality.
Finally, in section 6 we prove the weighted Trudinger inequality and Corollary \ref{embeddings}.

\section{Proof of the Isoperimetric inequality}\label{sec2}

In this section we prove the isoperimetric inequality with a monomial weight.
Our proof extends the one of the classical isoperimetric inequality due to the first author \cite{CSCM,CDCDS} (see also the last edition of \cite{Ch}).
In fact, setting $A=0$ in the following proof we obtain exactly the original one.
It is quite surprising (and fortunate) that this proof (which gives the best constant) can be adapted to the case of monomial weights.
A crucial fact in being able to obtain the sharp constant in the isoperimetric inequality is that
\[u(x)=|x|^2/2,\]
$x\in B_1\cap \R^n_*$, is the solution of
\begin{equation}\label{weak}
\left\{ \alignedat2 \div(x^A\nabla u) &= b_\Omega x^A
&\quad &\text{in } \Omega\\
x^A\frac{\partial u}{\partial\nu} &=x^A &\quad &\text{on }\partial \Omega ,
\endalignedat
\right.\end{equation}
for some constant $b_\Omega>0$ when $\Omega=B_1\cap \R^n_*$.

In a forthcoming paper \cite{CRS} we will use similar ideas to prove new sharp isoperimetric inequalities with homogeneous weights in open convex cones $\Sigma$ of $\R^n$.
We have already announced some of them in \cite{CRS-CRAS}.
Note that monomial weights are homogeneous functions in the convex cone $\Sigma=\R^n_*$.
In fact, the results in \cite{CRS} extend the present isoperimetric inequality with a monomial weight.

\begin{proof}[Proof of Theorem \ref{isoperimetric}.]
By symmetry, we can assume that $A=(A_1,...,A_k,0,...,0)$, with $A_i>0$ for $i=1,...,k$, where $0\leq k\leq n$.

Moreover, we can also suppose that $\Omega$ is contained in $\mathbb R^n_*$.
Indeed, split the domain $\Omega$ in at most $2^k$ disjoint subdomains $\Omega_j$, $j=1,...,J$, each one of them contained in the cone $\{\epsilon_ix_i>0,\ i=1,...,k\}$ for different $\epsilon_i\in\{-1,1\}$, and with $\overline\Omega=\overline\Omega_1\cup\cdots\cup\overline\Omega_J$.
Then, since the weight is zero on $\{x_i=0\}$ for each $i=1,...,k$, we have that $P(\Omega)=\sum_{j=1}^J P(\Omega_j)$ and $m(\Omega)=\sum_{j=1}^J m(\Omega_j)$.
Therefore
\[\frac{P(\Omega)}{m(\Omega)^{\frac{D-1}{D}}} \geq \min_{1\leq j\leq J}\left\{\frac{P(\Omega_j)}{m(\Omega_j)^{\frac{D-1}{D}}}\right\}=:\frac{P(\Omega_{j_0})}{m(\Omega_{j_0})^{\frac{D-1}{D}}},\]
with strict inequality unless $J=1$.
After some reflections, we may assume that $\Omega_{j_0}\subset \R^n_*$.
Moreover, since $\Omega_{j_0}$ is the intersection of a Lipschitz domain of $\R^n$ with $\R^n_*$, $\Omega_{j_0}$ can be approximated in weighted area and perimeter by smooth domains $\Omega_\varepsilon$ with $\overline\Omega_\varepsilon\subset \Omega_{j_0}\subset\R^n_*$.

Therefore, from now on we assume:
\[\Omega\ \textrm{is smooth and}\ \overline\Omega\subset \mathbb R^n_*.\]
In particular, $x^A\geq c$ in $\overline \Omega$ for some positive constant $c$.

Let $u$ be a solution of the Neumann problem
\begin{equation}
\left\{ \alignedat2 \div(x^A\nabla u) &= b_\Omega x^A
&\quad &\text{in } \Omega\\
\frac{\partial u}{\partial\nu} &=1 &\quad &\text{on }\partial \Omega ,
\endalignedat
\right. \label{eqsem}
\end{equation}
where the constant $b_\Omega$ is chosen so that the problem has a unique solution up to an additive constant, i.e.,
\begin{equation}\label{cttb}
b_\Omega=\frac{P(\Omega)}{m(\Omega)}.\end{equation}
Since the equation in \eqref{eqsem},
\begin{equation}
x^{-A}\div(x^A\nabla u) = \Delta u+\frac{A_1}{x_1}u_{x_1}+\cdots+\frac{A_n}{x_n}u_{x_n}=b_\Omega\label{eqn}
\end{equation}
is uniformly elliptic in $\Omega$, $u$ is smooth in $\overline \Omega$.
The $C^{1,1}$ regularity of $u$ up to $\overline\Omega$ will be crucial in the rest of the proof.

The following comment is not necessary to complete the proof, but it is useful to notice it here.
Problem \eqref{eqsem} is equivalent to \eqref{weak} since $\partial\Omega\subset\R^n_*$.
At the same time, when $\Omega=B_1^*=B_1\cap\R^n_*$ the solution to \eqref{weak} is given by $u(x)=|x|^2/2$, and we will have that all inequalities in the rest of the proof are equalities
for $\Omega=B_1^*$ (see Remark \ref{remparabola} for more details).

Coming back to the solution $u$ of \eqref{eqsem}, consider the lower contact set of $u$, defined by
\[\Gamma_u =\{ x \in \Omega \ : \ u(y) \ge u(x) + \nabla u (x) \cdot (y-x)\  \text{ for all } y \in \overline \Omega \}.\]
It is the set of points where the tangent hyperplane to the graph of $u$ lies below $u$ in all $\overline \Omega$.
Define also
\[\Gamma_u^*=\{x\in\Gamma_u\,:\, u_{x_1}(x)>0,...,u_{x_k}(x)>0\}=\Gamma_u\cap (\nabla u)^{-1}(\R^n_*).\]
We claim that
\begin{equation} \label{claim}
B_1^*\subset \nabla u (\Gamma_u^*),
\end{equation}
where $B_1^*=B_1(0)\cap \R^n_*$.

To show \eqref{claim}, take any $p\in \mathbb R^n$ satisfying $\vert p \vert <1$. Let $x\in \overline \Omega$ be a point such that
$$
\min_{y\in \overline \Omega} \,\{ u(y) -p\cdot y \} = u(x)-p\cdot x
$$
(this is, up to a sign, the Legendre transform of $u$).
If $x\in \partial \Omega$ then the exterior normal derivative of $u(y)-p\cdot y$ at $x$
would be nonpositive and hence $(\partial u /\partial\nu) (x) \leq p\cdot\nu\leq |p|<1$, a contradiction with \eqref{eqsem}.
It follows that $x\in
\Omega$ and, therefore, that $x$ is an interior minimum of the function $u(y)-p\cdot y$.
In particular, $p=\nabla u (x)$ and $x\in \Gamma_u$.
Thus $B_1\subset \nabla u(\Gamma_u)$.
Intersecting now both sides of this inclusion with $\R^n_*$, claim \eqref{claim} follows.
It is interesting to visualize geometrically the proof of the claim \eqref{claim}, by considering the graphs of the
functions $p\cdot y + c$ for $c\in \mathbb R$. These are parallel hyperplanes which lie, for $c$ close to $-\infty$, below the graph of $u$.
We let $c$ increase and consider the first $c$ for which there is contact or ``touching'' at a point $x$.
It is clear geometrically that
$x\not\in\partial\Omega$, since $\vert p\vert <1$ and $\partial u /\partial\nu =1$ on $\partial\Omega$.

Now, from \eqref{claim} we deduce
\begin{equation}\begin{split}
m(B_1^*) &\leq \int_{\nabla u (\Gamma_u^*)}p^A dp \leq \int_{\Gamma_u^*}(\nabla u(x))^A \det D^2u (x)dx\\
&=\int_{\Gamma_u^*}\frac{(\nabla u(x))^A}{x^A}\det D^2u (x)x^Adx. \label{ineq}\end{split}
\end{equation}
We have applied the area formula to the smooth map $\nabla u : \Gamma_u^* \rightarrow \mathbb R^n$, and we have used that its Jacobian, $\det D^2u$, is
nonnegative in $\Gamma_u$ by definition of this set.

We use now the weighted version of the arithmetic-geometric mean inequality,
\begin{equation}\label{weightedAM-GM}w_1^{\lambda_1}\cdots w_m^{\lambda_m}\leq \left(\frac{\lambda_1w_1+\cdots+\lambda_mw_m}{\lambda_1+\cdots+\lambda_m}\right)^{\lambda_1+\cdots+\lambda_m}.\end{equation}
Here $\lambda_i$ and $w_i$ are arbitrary nonnegative numbers.
To prove this inequality, take logarithms on both sides and use the concavity of the logarithm.
We apply \eqref{weightedAM-GM} to the numbers $w_i=u_{x_i}/x_i$ and $\lambda_i=A_i$ for $i=1,...,k$, and to the eigenvalues of
$D^2u(x)$ and $\lambda_j=1$ for $j=k+1,...,k+n$.
They are all nonnegative when $x\in \Gamma_u^*$.
We obtain
\[\left(\frac{u_{x_1}}{x_1}\right)^{A_1}\cdots\left(\frac{u_{x_k}}{x_k}\right)^{A_k}\det D^2u\leq
\left(\frac{A_1\frac{u_{x_1}}{x_1}+\cdots+A_k\frac{u_{x_k}}{x_k}+\Delta u}{A_1+\cdots+A_k+n}\right)^{A_1+\cdots+A_k+n}\ \ \textrm{in}\ \Gamma_u^*.\]

This, combined with \eqref{eqn}
\[A_1\frac{u_{x_1}}{x_1}+\cdots+A_k\frac{u_{x_k}}{x_k}+\Delta u=\frac{\div(x^A\nabla u)}{x^A} \equiv b_\Omega,\]
yields
\[\int_{\Gamma_u^*}\frac{(\nabla u(x))^A}{x^A}\det D^2u (x)x^Adx\leq \int_{\Gamma_u^*}\left(\frac{b_\Omega}{D}\right)^Dx^A dx.\]
Therefore, by \eqref{ineq} and \eqref{cttb},
\[m(B_1^*) \leq \left( \frac{P(\Omega)} {Dm(\Omega)} \right)^D m(\Gamma_u^*) \leq \left( \frac{P(
\Omega)} {Dm( \Omega )} \right)^D m(\Omega).\]
Thus, we conclude that
\begin{equation}
Dm(B_1^*)^{\frac{1}{D}} \leq \frac{P(\Omega)}{m(\Omega)^{\frac{D-1}{D}}}.\label{isopfin}
\end{equation}

Finally, an easy computation --- using that $|x|^2/2$ solves \eqref{weak} with $b_\Omega=D$ in $\Omega=B_1^*$ --- gives $P(B_1^*) = Dm(B_1^*)$.
Thus,
\begin{equation}\label{eqmp}Dm(B_1^*)^{\frac{1}{D}}=P(B_1^*)/m(B_1^*)^{\frac{D-1}{D}}\end{equation}
and the isoperimetric inequality \eqref{isop} follows.
\end{proof}

\begin{rem}\label{remparabola}
An alternative (and more instructive) way to finish the proof goes as follows.
When $\Omega=B_1^*$ we consider $u(x)=|x|^2/2$ and $\Gamma_u=B_1^*$.
Now, $\partial u/\partial\nu=1$ is only satisfied on $\R^n_*\cap \partial\Omega$ but, since $x^A\equiv 0$ on $\partial \R^n_*\cap \partial\Omega$, we have $b_{B_1^*} = P(B_1^*)/m(B_1^*)$ --- as in  \eqref{cttb}.
This is because $|x|^2/2$ solves problem $\div(x^A\nabla u)=b_\Omega x^{A}$ in $\Omega$, $x^Au_\nu=x^A$ on $\partial\Omega$ for $\Omega=B_1^*$.
For these concrete $\Omega$ and $u$ one verifies that all inequalities in the proof are equalities, and therefore from \eqref{isopfin} we deduce the isoperimetric inequality \eqref{isop}.
\end{rem}

\begin{rem}\label{grig} The fact that $P(\Omega)/m(\Omega)^{\frac{D-1}{D}}\geq c$ for some nonoptimal constant $c$ is an interesting consequence of the following result of A. Grigor'yan \cite{G} (see also \cite{M2}).

We say that a manifold $M$ satisfies the $m$-isoperimetric inequality if there exists a positive constant $c$ such that $\mu(\partial\Omega)\geq c\mu(\Omega)^{\frac{m-1}{m}}$ for each $\Omega \subset M$.
In \cite{G}, he proved that if $M_1$ and $M_2$ are manifolds that satisfy the $m_1$-isoperimetric and $m_2$-isoperimetric inequalities, respectively, then the product manifold $M_1\times M_2$ satisfies the $(m_1+m_2)$-isoperimetric inequality. By applying this result to $M_i=(\mathbb R,x_i^{A_i}dx_i)$, this allows us to reduce the problem to $n=1$, and in this case the isoperimetric inequality is easy to verify.
\end{rem}

\section{Weighted Sobolev inequality}
\label{sec3}

The aim of this section is to prove the Sobolev inequality with a monomial weight, that is, part (a) of Theorem \ref{sob}.

As in the classical inequality in $\mathbb R^n$, we can deduce any weighted Sobolev inequality from the isoperimetric inequality with the same weight via the coarea formula.
Moreover, if the isoperimetric inequality has the sharp constant then this procedure gives the optimal constant for the Sobolev inequality when the exponent is $p=1$ (see the following proof and also Remark \ref{optconst1}).
This classical argument is valid even on Riemannian manifolds; see for example \cite{Ch}.
We use it to prove part (a) of Theorem \ref{sob}.

\begin{proof}[Proof of Theorem \ref{sob} (a)] We prove first the case $p=1$.
By density arguments, we can assume $u\geq0$ and also $u\in C^\infty_c(\R^n)$.
Moreover, by approximation we can suppose $u\in C^\infty_c(\R^n_*)$.
Indeed, consider $\tilde u_\varepsilon= u\eta_\varepsilon$, where $\eta_\varepsilon\in C^\infty_c(\R^n_*)$ is a function satisfying $\eta_\varepsilon\equiv1$ in the set $\{x_i>\varepsilon\ \textrm{whenever}\ A_i>0\}$ and $|\nabla \eta_\varepsilon|\leq C/\varepsilon$.
Then, it is clear that
\[\|u\eta_\varepsilon\|_{L^{\frac{D}{D-1}}(\R^n_*,x^Adx)}\longrightarrow \|u\|_{L^{\frac{D}{D-1}}(\R^n_*,x^Adx)}\]
as $\varepsilon\rightarrow0$.
Moreover,
\[ \|\nabla \eta_\varepsilon\|_{L^1(\R^n_*,x^Adx)}\leq \sum_{A_i>0}\int_{ \{0\leq x_i\leq \varepsilon\} }\frac{C}{\varepsilon}x^Adx\leq \sum_{A_i>0} C\varepsilon^{A_i}\longrightarrow 0,\]
and thus
\[\|\nabla(u \eta_\varepsilon)\|_{L^1(\R^n_*,x^Adx)}\longrightarrow \|\nabla u\|_{L^1(\R^n_*,x^Adx)}.\]

Thus, we now have $u\in C^\infty_c(\R^n_*)$.
For each $t\geq0$, define
\[\{u>t\}:=\{x\in\R^n_*\,:\, u(x)>t\}\ \ \ \textrm{and}\ \ \ \{u=t\}:=\{x\in\R^n_*\,:\, u(x)=t\}.\]
By Theorem \ref{isoperimetric} and Sard's Theorem, we have
\begin{equation}\label{isoput}
m(\{u>t\})^{\frac{D-1}{D}}\leq C_1P(\{u>t\})=C_1\int_{\{u=t\}}x^Ad\sigma
\end{equation}
for almost all $t$ (those $t$ for which $\{u=t\}$ is smooth).
Here, $C_1$ is the optimal constant in \eqref{isop}, i.e., recalling \eqref{eqmp}
\begin{equation}\label{cttC1}
C_1=\frac{P(B_1^*)}{m(B_1^*)^{\frac{D-1}{D}}}=Dm(B_1^*)^{\frac1D}.
\end{equation}

Letting $\chi_A$ be the characteristic function of the set $A$, we have
\[u(x)=\int_0^{+\infty}\chi_{\{u(x)>\tau\}}d\tau.\]

Thus, by Minkowski's integral inequality
\begin{eqnarray*}\left(\int_{\R^n_*}x^Au^{\frac{D}{D-1}}dx\right)^{\frac{D-1}{D}}&\leq&
 \int_0^{+\infty}\left(\int_{\R^n_*}\chi_{\{u(x)>\tau\}}x^Adx\right)^{\frac{D-1}{D}}d\tau\\
 &=&\int_0^{+\infty}m(\{u>\tau\})^{\frac{D-1}{D}}d\tau.\end{eqnarray*}
Inequality \eqref{isoput}, together with the coarea formula, yield
\[\left(\int_{\R^n_*}x^Au^{\frac{D}{D-1}}dx\right)^{\frac{D-1}{D}}\leq c_0\int_0^{+\infty}\int_{\{u=t\}}x^Ad\sigma\,d\tau=c_0\int_{\R^n_*}x^A|\nabla u|dx,\]
and the theorem is proved for $p=1$.

It remains to prove the case $1<p<D$.
Take $u\in C^1_c(\R^n)$, and define $v=|u|^{\gamma}$, where $\gamma=\frac{p_*}{1_*}$. Since, $\gamma>1$, we have $v\in C^1_c(\R^n)$, and we can apply the weighted Sobolev inequality with exponent $p=1$ (proved above) to get
\[\left(\int_{\R^n_*}x^A|u|^{p_*}dx\right)^{1/1_*}=\left(\int_{\R^n_*}x^A|v|^{\frac{D}{D-1}}dx\right)^{\frac{D-1}{D}}\leq c_0\int_{\R^n_*}x^A|\nabla v|dx.\]
Now, $|\nabla v|=\gamma |u|^{\gamma-1}|\nabla u|$, and by H\"older's inequality we deduce
\[\int_{\R^n_*}x^A|\nabla v|dx\leq C\left(\int_{\R^n_*}x^A|\nabla u|^pdx\right)^{1/p}\left(\int_{\R^n_*}x^A|u|^{(\gamma-1)p'}dx\right)^{1/p'}.\]
Finally, from the definition of $\gamma$ and $p_*$ it follows that
\[\frac{1}{1_*}-\frac{1}{p'}=\frac{1}{p_*},\qquad (\gamma-1)p'=p_*,\]
and hence,
\[\left(\int_{\R^n_*}x^A|u|^{p_*}dx\right)^{1/p_*}\leq C\left(\int_{\R^n_*}x^A|\nabla u|^{p}dx\right)^{1/p}.\]
\end{proof}

\begin{rem}\label{optconst1}
Since the constant appearing in \eqref{isoput} is optimal, this proof gives the optimal constant for the weighted Sobolev inequality for $p=1$.
This is because for each Lipschitz open set $E$ there exists an increasing sequence of smooth functions $u_\varepsilon\rightarrow\chi_E$, such that
$\|\nabla u_\varepsilon\|_{L^1(\R^n_*,x^Adx)}\rightarrow P(E)$.

Moreover, for $p=1$ it follows from the previous proof (in fact from the use of Minkowski's inequality) that if equality is attained by a function $u$, then all the sets $\{u>t\}$ must coincide for $t\in (0,\max u)$.
That is, the extremal function must be a characteristic function.
This proves that the optimal constant is not attained by any $W^{1,1}_0(\R^n,x^Adx)$ function for $p=1$.
\end{rem}

We give now an alternative and short proof of part (a) of Theorem \ref{sob} --- without best constant --- under some additional assumptions.
Indeed, under the hypotheses $A_i>0$ for all $i$ and $u_{x_i}\leq0$ in $\{x_i>0,\ i=1,...,n\}$, we establish the weighted Sobolev inequality \eqref{sobolev} following the ideas used in \cite{CR} to prove the isoperimetric inequality in dimension $n=2$ (without best constant) with the weight $\sigma^a\tau^b$.
The following proof is much more elementary than the previous one, which used the weighted isoperimetric inequality.
It does not use any elliptic problem nor the coarea formula, and it is also shorter.
However, it does not give the best constant in the inequality, even for $p=1$.
The monotonicity hypotheses $u_{x_i}\leq0$ in $\{x_i>0,\ i=1,...,n\}$ are equivalent to \eqref{assumption} in Proposition \ref{sobsigmatau}.
As said before, the weighted Sobolev inequality under these monotonicity assumptions suffices for some applications to nonlinear problems.

\begin{prop} Let $A$ be a positive vector in $\mathbb R^n$ and $1\leq p< D$ be a real number.
Then, there exists a constant $C$ such that for all $u\in C^{1}_c(\mathbb R^n)$ satisfying
\begin{equation}\label{ass}
u_{x_i}\leq0\ \textrm{ in }\ (\R_+)^n\ \textrm{ for }\ i=1,...,n,
\end{equation}
we have
\begin{equation*}\label{sobolevnonoptimal}
\left(\int_{(\R_+)^n}x^A|u|^{p_*}dx\right)^{1/p_*}\leq C\left(\int_{(\R_+)^n}x^A|\nabla u|^pdx\right)^{\frac1p},\end{equation*}
where $p_*=\frac{pD}{D-p}$ and $D=A_1+\cdots+A_n+n$.
\end{prop}

\begin{proof} It suffices to prove the case $p=1$, since the inequality for $1<p<D$ follows from it by applying H\"older's inequality --- see the previous proof of Theorem \ref{sob} (a).

From assumption \eqref{ass}, we deduce $u\geq0$ in $(\R_+)^n$.
Now, integrating by parts we have
\[\begin{split}
\int_{(\R_+)^n}x^A(|u_{x_i}|+\cdots+|u_{x_n}|)dx&=-\int_{(\R_+)^n}x^A(u_{x_1}+\cdots+u_{x_n})dx\\
&=\int_{(\R_+)^n} x^Au\left(\frac{A_1}{x_1}+\cdots+\frac{A_n}{x_n}\right)dx,\end{split}\]
and thus
\begin{equation}\label{r1}
\int_{(\R_+)^n} x^Au\left(\frac{1}{x_1}+\cdots+\frac{1}{x_n}\right)dx\leq K\int_{(\R_+)^n} x^A|\nabla u|dx,\end{equation}
where $K=\sqrt n/\min_i A_i$.

Let now $\lambda>0$ be such that
\[\int_{(\R_+)^n}x^A u^{\frac{D}{D-1}}dx=b\lambda^D,\]
where $b=\int_{\{0\leq x_i\leq 1\}} x^{A}dx$.
Here $\{0\leq x_i\leq 1\}=\{x\in\R^n\,:\, 0\leq x_i\leq 1\ \textrm{for}\ i=1,...,n\}$.

We claim that, for each $x\in(\R_+)^n$ we have $u(x)^{\frac{1}{D-1}}\leq \frac{\lambda}{x_i}$ for some $i\in\{1,...,n\}$.
Indeed, otherwise there would exist
$y\in(\R_+)^n$ such that $u(y)^{\frac{1}{D-1}}>\frac{\lambda}{y_i}$ for each $i$, and therefore
\[u(y)^{\frac{D}{D-1}}>\frac{\lambda^D}{y^{A+1}},\]
where $A+1=A+(1,...,1)=(A_1+1,...,A_n+1)$.
But, by \eqref{ass}, $u(x)\geq u(y)$ if $0\leq x_i\leq y_i$ for all $i=1,...,n$.
We deduce
\[\int_{\{0\leq x_i\leq y_i\}}x^Au(x)^{\frac{D}{D-1}}dx>\lambda^D
\int_{\{0\leq x_i \leq y_i\}}x^Ay^{-A-1}dx=\lambda^D\int_{\{0\leq z_i\leq 1\}}z^Adz
=b\lambda^D,\]
a contradiction.

Hence,
\[u(x)^{\frac{1}{D-1}}\leq \lambda\left(\frac{1}{x_1}+\cdots+\frac{1}{x_n}\right)\qquad\textrm{in}\ (\R_+)^n,\]
and therefore
\begin{equation}\label{r2}
\int_{(\R_+)^n} x^Au^{\frac{D}{D-1}}dx\leq \lambda \int_{(\R_+)^n} x^Au\left(\frac{1}{x_1}+\cdots+\frac{1}{x_n}\right)dx.\end{equation}
Finally, taking into account the value of $\lambda$
\[\lambda=b^{-\frac1D}\left(\int_{(\R_+)^n} x^Au^{\frac{D}{D-1}}dx\right)^{\frac1D},\]
we deduce from \eqref{r2} and \eqref{r1} that
\[\begin{split}\left(\int_{(\R_+)^n} x^Au^{\frac{D}{D-1}}dx\right)^{\frac{D-1}{D}}&\leq b^{-\frac1D}\int_{(\R_+)^n} x^Au\left(\frac{1}{x_1}+\cdots+\frac{1}{x_n}\right)dx\\
&\leq K b^{-\frac1D}\int_{(\R_+)^n} x^A|\nabla u|dx.\end{split}\]
This completes the proof and gives as constant $Kb^{-\frac1D}$, computed explicitly within the proof.
\end{proof}

This proof can not be used to establish the classical Sobolev inequality.
Indeed, the constant on the right hand side blows up as $A_i\rightarrow0$ for some $i$.
It is surprising that the above proof of the Sobolev inequality with the monomial weight $x^A$, $A>0$, seems more elementary than those of the classical Sobolev without weight.

The following remark justifies our assumption $A\geq0$ in the weighted Sobolev inequality \eqref{sobolev}.
It is related to the monotonicity assumption \eqref{assumption} in Proposition \ref{sobsigmatau}.

\begin{rem}\label{remass}
When $a<0$ or $b<0$ inequality \eqref{sobsigmatau} is not valid without the monotonicity assumption \eqref{assumption}.
To prove it, we only need to take functions $u$ with support away from the origin, as follows.
Assume that $a<0$, $a+b>0$ (and thus $b>0$), and that \eqref{sobsigmatau} holds for functions $u$ with support in the ball $B_1(x_0)$, with $x_0=(2,0)$.
Then, since $\sigma^a$ is bounded in this ball from above and below by positive constants, the same inequality holds --- with a larger constant $C$ --- with the weight $\sigma^a\tau^b$ replaced by $\tau^b$.
But, since $a<0$, we have $q':=\frac{2D'}{D'-2}<\frac{2D}{D-2}$, where $D'=b+2$.
This is a contradiction with the fact that the exponent $q'$ is optimal for the weight $\tau^b$ (which can be seen by a scaling argument, i.e., considering the rescaled functions $u_{\lambda}(x)=u(x_0+\lambda(x-x_0))$, with $\lambda\geq1$).
Of course, when $a$ and $b$ are both nonnegative this argument does not work.
\end{rem}

\begin{rem}
One can think on adapting the classical proof of the Sobolev inequality by Gagliardo and Nirenberg (see for example \cite{E}) to the case of monomial weights.
As we show next, this leads to inequality
\begin{equation}\label{evans}
\left(\int_{\mathbb R^n}x^{A}|u|^{\frac{n}{n-1}}dx\right)^{\frac{n-1}{n}}\leq \int_{\mathbb R^n}x^{\frac{n-1}{n}A}|\nabla u|dx,
\end{equation}
but not to our Sobolev inequality \eqref{sobolev} with the same weight $x^A$ in both integrals.
The constant $C$ (which does not appear) on the right hand side equals $1$.
To prove \eqref{evans}, one shows first that
\begin{equation}\label{evans2}
|x_i|^{\frac{n-1}{n}A_i}|u(x)|\leq \int_{\mathbb R}|y_i|^{\frac{n-1}{n}A_i}|\nabla u(x_1,...,x_{i-1},y_i,x_{i+1},...,x_n)|dy_i.
\end{equation}
This follows by integrating $u_{y_i}$ on $(x_i,+\infty)$ if $x_i>0$ and on $(-\infty,x_i)$ if $x_i<0$, and using $|x_i|\leq|y_i|$ in these halflines.
Then, \eqref{evans2} yields
\[|x_1|^{\frac{A_1}{n}}\cdots|x_n|^{\frac{A_n}{n}}|u(x)|^{\frac{n}{n-1}}\leq \prod_{i=1}^n\left(\int_{-\infty}^{+\infty}|\nabla u(x_1,...,y_i,...,x_n)||y_i|^{\frac{n-1}{n}A_i}dy_i\right)^{\frac{1}{n-1}}.\]
Integrating both sides with respect to the measure $x^{\frac{n-1}{n}A}dx$ we deduce
\[\int_{\R^n}x^A|u(x)|^{\frac{n}{n-1}}dx\leq \int_{\R^n}\prod_{i=1}^n\left(\int_{-\infty}^{+\infty}|\nabla u(x_1,...,y_i,...,x_n)||y_i|^{\frac{n-1}{n}A_i}dy_i\right)^{\frac{1}{n-1}}x^{\frac{n-1}{n}A}dx,\]
and the proof of \eqref{evans} is completed in the same way as the classical one with the measures $dx_i$ and $dy_i$ replaced by $d\mu_i(x_i)=|x_i|^{\frac{n-1}{n}A_i}dx_i$ and $d\mu_i(y_i)=|y_i|^{\frac{n-1}{n}A_i}dy_i$.

Different from \eqref{sobolev}, inequality \eqref{evans} is the Sobolev inequality for the Riemannian manifold conformal to $\R^n$ with conformal factor $g=x^A$.
Indeed, the Riemannian gradient in $\mathbb R^n$ with this metric is given by $\nabla_R u=x^{-\frac An}\nabla u$, and hence it holds
\[x^{\frac{n-1}{n}A}|\nabla u|=x^A|\nabla_R u|.\]
Moreover, from this Sobolev inequality one can deduce the following isoperimetric inequality (with nonoptimal constant) on this manifold
\[\left(\int_\Omega x^Adx\right)^{\frac{n-1}{n}}\leq\int_{\partial\Omega}x^{\frac{n-1}{n}A}d\sigma.\]
\end{rem}

To end this section, we give an immediate consequence of Theorem \ref{sob}.
Recall that in \cite{CR} we wanted to prove inequality \eqref{ddf} for $n=2$ and that, after a change of variables, we saw that it is equivalent to the Sobolev inequality \eqref{sobolev} with a monomial weight.

\begin{cor}\label{cor} Let $\alpha_1,...,\alpha_n$ be real numbers such that $\alpha_i\in[0,1)$.
There exists a constant $C$ such that for all $u\in C^1_c(\mathbb R^n)$,
\begin{equation}\label{ddf}
\left(\int_{\mathbb R^n_*}|u|^{p_*}dx\right)^{\frac{1}{p_*}}\leq C\left(\int_{\mathbb R^n_*}\bigl\{|x_1|^{p\alpha_1}|u_{x_1}|^p+\cdots+|x_n|^{p\alpha_n}|u_{x_n}|^p\bigr\}dx\right)^{\frac1p},\end{equation}
where $p_*=\frac{pD}{D-p}$ and $D=n+\frac{\alpha_1}{1-\alpha_1}+\cdots+\frac{\alpha_n}{1-\alpha_n}$.
\end{cor}

\begin{proof} It suffices to make the change of variables $y_i=x_i^{1-\alpha_i}$ in \eqref{ddf} and then apply Theorem \ref{sob} with $A_i=\frac{\alpha_i}{1-\alpha_i}$.
\end{proof}

The optimal exponent in \eqref{ddf} is $p_*=\frac{pD}{D-p}$, as in \eqref{sobolev}.
However, in \eqref{ddf} the constant $D$ has no clear interpretation in terms of any ``dimension''.

\section{Best constant and extremal functions in the weighted Sobolev inequality}
\label{sec4}

In this section we obtain the best constant and the extremal functions in the weighted Sobolev inequality \eqref{sobolev}.

The first step is to compute the measure of the unit ball in $\R^n_*$ with the weight $x^A$.
From this, we will obtain the optimal constant in the isoperimetric inequality and, therefore, the optimal constant in Sobolev inequality for $p=1$ (see Remark \ref{optconst1}).

\begin{lem}\label{lema4.1} Let $A$ be a nonnegative vector in $\R^n$ and $B_1^*=B_1(0)\cap\R^n_*$. Then,
\[\int_{B_1^*} x^Adx=\frac{\Gamma\left(\frac{A_1+1}{2}\right)\Gamma\left(\frac{A_2+1}{2}\right)\cdots \Gamma\left(\frac{A_n+1}{2}\right)}{2^k\Gamma\left(1+\frac D2\right)},\] where $D=A_1+\cdots+A_n +n$ and $k$ is the number of strictly positive entries of $A$.
\end{lem}

\begin{proof} We will prove by induction on $n$ that
\[\int_{B_1} x^Adx=\frac{\Gamma\left(\frac{A_1+1}{2}\right)\Gamma\left(\frac{A_2+1}{2}\right)\cdots \Gamma\left(\frac{A_n+1}{2}\right)}{\Gamma\left(1+\frac D2\right)},\]
where $B_1$ is the unit ball in $\mathbb R^n$.
After this, the the lemma follows by taking into account that $m(B_1^*)=m(B_1)/2^k$.

For $n=1$ it is immediate.
Assume that this is true for $n-1$ and let us prove it for $n$. Let us denote $x=(x',x_n)$, $A=(A',A_n)$, with $x',A'\in \mathbb R^{n-1}$, and $D'=A_1+\cdots+A_{n-1}+n-1$.
Then,
\begin{eqnarray*} \int_{B_1} x^Adx&=& \int_{-1}^1|x_n|^{A_n}\left(\int_{|x'|\leq \sqrt{1-x_n^2}}x'^{A'}dx'\right)dx_n\\
&=& \int_{-1}^1|x_n|^{A_n}\left((1-x_n^2)^{\frac{D'}{2}}\int_{|y'|\leq 1}y'^{A'}dy'\right)dx_n\\
&=& \int_{|y'|\leq 1}y'^{A'}dy'\int_{-1}^1|x_n|^{A_n}(1-x_n^2)^{\frac{D'}{2}}dx_n,
\end{eqnarray*}
and hence it remains to compute $\int_{-1}^1|x_n|^{A_n}(1-x_n^2)^{\frac{D'}{2}}dx_n$.

Making the change of variables $x_n^2=t$ one obtains
\begin{eqnarray*} \int_{-1}^1|x_n|^{A_n}(1-x_n^2)^{\frac{D'}{2}}dx_n&=&2\int_0^1x_n^{A_n}(1-x_n^2)^{\frac{D'}{2}}dx_n\\
&=& \int_0^1 t^{\frac{A_n-1}{2}}(1-t)^{\frac{D'}{2}}dt\\
&=& B\left(\frac{A_n+1}{2},1+\frac{D'}{2}\right),\end{eqnarray*}
where $B$ is the Beta function.
Now, since
\begin{equation}\label{beta}
B(p,q)=\frac{\Gamma(p)\Gamma(q)}{\Gamma(p+q)},\end{equation}
then
\begin{eqnarray*}\int_{B_1} x^Adx&=&\int_{|y'|\leq 1}y'^{A'}dy'\int_{-1}^1x_n^{A_n}(1-x_n^2)^{\frac{D'}{2}}dx_n\\
&=&\frac{\Gamma\left(\frac{A_1+1}{2}\right)\cdots \Gamma\left(\frac{A_{n-1}+1}{2}\right)}{\Gamma\left(1+\frac{D'}{2}\right)}\cdot \frac{\Gamma\left(\frac{A_n+1}{2}\right)\Gamma\left(1+\frac{D'}{2}\right)}{\Gamma\left(1+\frac{D}{2}\right)}\\
&=& \frac{\Gamma\left(\frac{A_1+1}{2}\right)\Gamma\left(\frac{A_2+1}{2}\right)\cdots \Gamma\left(\frac{A_n+1}{2}\right)}{\Gamma\left(1+\frac D2\right)},\end{eqnarray*}
and the lemma follows.
\end{proof}

Now, as in the classical Sobolev inequality, we find the extremal functions in our weighted Sobolev inequality by reducing it to the radial case.
To do this, we use a weighted version of a rearrangement inequality due to Talenti \cite{T}.
His result states that, whenever balls minimize the isoperimetric quotient with a weight $w$, there exists a radial rearrangement (of $u$) which preserves $\int f(u)w\,dx$ and decreases $\int \Phi(|\nabla u|)w\,dx$ (under some conditions on $\Phi$).
When $w=x^A$, this is stated in the following.

\begin{prop}\label{rear} Let $u$ be a Lipschitz continuous function in $\mathbb R^n_*$ with compact support in $\overline{\mathbb R^n_*}$. Then, denoting $m(E)=\int_E x^Adx$, there exists a radial rearrangement $u_*$ of $u$ such that
\begin{itemize}
\item[(i)] $m(\{|u|>t\})=m(\{u_*>t\})$ for all $t$,
\item[(ii)] $u_*$ is radially decreasing,
\item[(iii)] for every Young function $\Phi$ (i.e., convex and increasing function that vanishes at $0$),
\[\int_{\mathbb R^n_*}\Phi(|\nabla u_*|)x^Adx\leq \int_{\mathbb R^n_*} \Phi(|\nabla u|)x^Adx.\]
\end{itemize}
\end{prop}

\begin{proof} It is a direct consequence of the main theorem in \cite{T} and our isoperimetric inequality \eqref{isop}.
\end{proof}

We can now find the best constant in the weighted Sobolev inequality \eqref{sobolev}. The proof is based on Proposition \ref{rear}, which allows us to reduce the problem to radial functions in $\mathbb R^n_*$.
Then, the functional that we must minimize is exactly the same as in the classical Sobolev inequality but with a noninteger exponent $D$ in the 1D weight, and the proof finishes by applying another result of Talenti in \cite{T2}.

\begin{prop}\label{plp} The best constant in the Sobolev inequality \eqref{sobolev} is given by
\begin{equation} \label{bestcnt} C_1= D\left(\frac{\Gamma\left(\frac{A_1+1}{2}\right)\Gamma\left(\frac{A_2+1}{2}\right)\cdots \Gamma\left(\frac{A_n+1}{2}\right)}{2^k\Gamma\left(1+\frac D2\right)}\right)^{\frac1D}\qquad \textrm{for }\ p=1\end{equation}
and by
\begin{equation}\label{bestcnt2} C_p=C_1D^{\frac{1}{D}-1-\frac1p}\left(\frac{p-1}{D-p}\right)^{\frac{1}{p'}}\left(\frac{p'\Gamma(D)}{\Gamma\left(\frac{D}{p}\right)\Gamma\left(\frac{D}{p'}\right)}\right)^{\frac1D}\qquad \textrm{for }\ 1<p<D.\end{equation}
Here, $p'=\frac{p}{p-1}$ and $k$ is the number of positive entries in the vector $A$.

Moreover, this constant is not attained by any function in $W^{1,1}_0(\R^n,x^Adx)$ when $p=1$.
Instead, when $1<p<D$ this constant is attained in $W^{1,p}_0(\R^n,x^Adx)$ by
\[u_{a,b}(x)=\left(a+b|x|^{\frac{p}{p-1}}\right)^{1-\frac{D}{p}},\] where $a$ and $b$ are arbitrary positive constants.
\end{prop}

Before giving the proof of Proposition \ref{plp}, we recall Lemma 2 from \cite{T2}, where the best constant for the classical Sobolev inequality is obtained.

\begin{lem}[\cite{T2}]\label{t2}
Let $m$, $p$, and $q$ be real numbers such that \[1<p<m\qquad\textrm{and}\qquad q=\frac{mp}{m-p}.\]
Let $u$ be any real-valued function of a real variable $r$, which is Lipschitz continuous and such that
\[\int_0^{+\infty}r^{m-1}|u'(r)|^pdr<+\infty\qquad\textrm{and}\qquad u(r)\rightarrow0\ \ \mbox{as}\ \ r\rightarrow+\infty.\]
Then,
\[\frac{\left(\int_0^{+\infty} r^{m-1}|u(r)|^{q}dr\right)^{\frac1q}}{\left(\int_0^{+\infty} r^{m-1}|u'(r)|^{p}dr\right)^{\frac1p}}\leq
\frac{\left(\int_0^{+\infty} r^{m-1}|\varphi(r)|^{q}dr\right)^{\frac1q}}{\left(\int_0^{+\infty} r^{m-1}|\varphi'(r)|^{p}dr\right)^{\frac1p}}=:J(\varphi),\]
where $\varphi$ is any function of the form
\[\varphi(r)=(a+br^{p'})^{1-\frac mp}\] with $a$ and $b$ positive constants. Here $p'=p/(p-1)$.

Moreover,
\[J(\varphi)=m^{-\frac1p}\left(\frac{p-1}{m-p}\right)^{\frac{1}{p'}}\left[\frac{1}{p'}B\left(\frac mp,\frac{m}{p'}\right)\right]^{-\frac1m},\]
where $B$ is the Beta function.
\end{lem}

We can now give the

\begin{proof}[Proof of Proposition \ref{plp}.]
For $p=1$, the best constant in Sobolev inequality is the same than in the isoperimetric inequality (see Remark \ref{optconst1}).
Recalling \eqref{cttC1}, it is given by $C_1=Dm(B_1^*)^{1/D}$.
Thus, the value of $C_1$ follows from Lemma \ref{lema4.1}.
That $C_1$ is not attained by any $W^{1,1}_0(\R^n,x^Adx)$ function was explained in Remark \ref{optconst1}.

Let now $1<p<D$, $u$ be a $C^1(\overline{\R^n_*})$ function with compact support in $\overline{\R^n_*}$, and $u_*$ be its radial rearrangement given by Proposition \ref{rear}.
Then, by the proposition,
\[\frac{\|\nabla u_*\|_{L^{p}(\mathbb R^n_*,x^Adx)}}{\|u_*\|_{L^{p_*}(\mathbb R^n_*,x^Adx)}}\leq \frac{\|\nabla u\|_{L^{p}(\mathbb R^n_*,x^Adx)}}{\|u\|_{L^{p_*}(\mathbb R^n_*,x^Adx)}}.\]
Moreover,
\begin{eqnarray*} \int_{\mathbb R^n_*}x^A|u_*|^{p_*}dx&=& \int_0^{\infty}\left(\int_{\partial B_r^*}x^A|u_*|^{p_*}d\sigma\right)dr\\
&=& \int_0^{\infty} r^{D-1}|u_*|^{p_*}\left(\int_{\partial B_1^*}x^Ad\sigma\right)dr\\
&=&P(B_1^*)\int_0^{\infty} r^{D-1}|u_*|^{p_*}dr\end{eqnarray*}
and, analogously,
\[\int_{\mathbb R^n_*}x^A|\nabla u_*|^{p}dx=P(B_1^*)\int_0^{\infty} r^{D-1}|u_*'|^{p}dr.\]
Therefore, the best constant in the Sobolev inequality can be computed as
\begin{equation*}\label{cp}
\inf_{u\in C^1_c(\mathbb R^n)}\frac{\|\nabla u\|_{L^{p}(\mathbb R^n_*,x^Adx)}}{\|u\|_{L^{p_*}(\mathbb R^n_*,x^Adx)}}
=P(B_1^*)^{\frac{1}{D}}\inf_{u\in C^1_c(\mathbb R)}\frac{\left(\int_0^{\infty} r^{D-1}|u'|^{p}dr\right)^{1/p}}{\left(\int_0^{\infty} r^{D-1}|u|^{p_*}dr\right)^{1/p_*}},\end{equation*}
where we have used that $\frac1p-\frac{1}{p_*}=\frac1D$.
Recalling \eqref{eqmp} and \eqref{cttC1}, we have
\[P(B_1^*)^{\frac{1}{D}}=D^{\frac1D}m(B_1^*)^{\frac1D}=D^{\frac{1}{D}-1}C_1.\]
The value of $C_p$ follows from Lemma \ref{t2}, using \eqref{bestcnt} and \eqref{beta}.
From Lemma \ref{t2} it also follows that the functions $u_{a,b}$ in \eqref{extr} attain the best constant $C_p$.
\end{proof}

To end this section, we prove part (b) of Theorem \ref{sob}.

\begin{proof}[Proof of Theorem \ref{sob} (b).] For $p=1$ this was proved in Section \ref{sec3}; see Remark \ref{optconst1}.
For $p>1$ the result is proved in Proposition \ref{plp}.
\end{proof}

\section{Weighted Morrey inequality}\label{sec5}

In this section we prove Theorem \ref{morrey} \footnote{The proof of Theorem 1.6 in the previous version of this paper was not correct.
Indeed, Lemma~5.1 in that version was not true as stated therein. We thank Georgios Psaradakis for pointing this to us.},
that is, we show
\[\frac{|u(x)-u(y)|}{|x-y|^\alpha}\leq C\left(\int_{\R^n_*}|\nabla u|^p z^Adz\right)^{1/p},\qquad \alpha=1-\frac{D}{p},\]
for $p>D$.
The next lemma establishes this inequality for $y=0$.
Once this is done, Theorem \ref{morrey} follows quite easily from this case.

\begin{lem}\label{lemamorrey}
Let $A$ be a nonnegative vector in $\R^n$, $D=A_1+\cdots+A_n+n$, and $p>D$.
Let $u\in C^1_c(\mathbb R^n)$ and $x\in \mathbb R^n_*$.
Then,
\begin{equation}\label{ineq-morrey-0}
|u(x)-u(0)|\leq C\left(\int_{\R^n_*}|\nabla u|^p z^Adz\right)^{1/p} |x|^{1-\frac{D}{p}},
\end{equation}
where $C$ is a constant depending only on $p$ and $D$.
\end{lem}

Before proving this result, let us show briefly the ideas of the proof by establishing first the (weaker) inequality
\begin{equation}\label{easycase}
|u(x)-u(0)|\leq C\left(\int_{\R^n}|\nabla u|^p |z|^\gamma dz\right)^{1/p} |x|^{1-\frac{n+\gamma}{p}},\qquad p>D=n+|A|=:n+\gamma.
\end{equation}
Note that since $z^A\leq |z|^{|A|}$, then the inequality in Lemma \ref{lemamorrey} is stronger than \eqref{easycase}.

To show \eqref{easycase}, we use the classical Morrey inequality in $\R^n$.
Indeed, for any $x\in \R^n$ with $|x|=r$ we have (see Theorem 7.17 in \cite{GT})
\begin{equation}\label{andp}
\begin{split}|u(x)-u(x/2)|&\leq C\left(\int_{B_{r/2}(x)}|\nabla u|^p dz\right)^{1/p} r^{1-\frac{n}{p}}\\
&\leq C\left(\int_{\R^n}|\nabla u|^p |z|^\gamma dz\right)^{1/p} r^{1-\frac{n+\gamma}{p}},
\end{split}\end{equation}
where we have used that $|z|^\gamma\geq cr^\gamma$ in $B_{r/2}(x)$.

Thus, for each $x\in \R^n$ we have
\[|u(x)-u(x/2)|\leq C\left(\int_{\R^n}|\nabla u|^p |z|^\gamma dz\right)^{1/p} |x|^{1-\frac{n+\gamma}{p}}.\]
Writing the same inequality for $x/2$, $x/4$, $x/8$, etc., and adding up a geometric series, we then find
\[\begin{split}
|u(x)-u(0)|&\leq C\left(\int_{\R^n}|\nabla u|^p |z|^\gamma dz\right)^{1/p} \sum_{k\geq0}\left|\frac{x}{2^k}\right|^{1-\frac{n+\gamma}{p}}\\
&\leq C\left(\int_{\R^n}|\nabla u|^p |z|^\gamma dz\right)^{1/p}|x|^{1-\frac{n+\gamma}{p}}.\end{split}\]
This establishes \eqref{easycase}.

We now establish Lemma \ref{lemamorrey} by using the same idea.
The proof will be a little bit more involved because one has to deal with the weight $z^A=z_1^{A_1}\cdots z_n^{A_n}$, and thus we will need to add up $n$ geometric series.

\begin{proof}[Proof of Lemma \ref{lemamorrey}]
Recall that we want to show \eqref{ineq-morrey-0}.
For it, we assume without loss of generality that $x=(x_1,...,x_n)$, with $x_1\geq x_2\geq \cdots \geq x_n\geq0$.
Also, we denote $x'=(x_1,...,x_{n-1})$.

Then, for any number $y$ such that $0\leq y\leq x_{n-1}$, we have
\[\begin{split}
|u(x',y)-u(x',y/2)|&\leq C\left(\int_{B_{y/2}(x',y)}|\nabla u|^p dz\right)^{1/p} |y|^{1-\frac{n}{p}}\\
&\leq C\left(\int_{\R^n_*}|\nabla u|^p z^A dz\right)^{1/p} |y|^{1-\frac{D}{p}}.\end{split}\]
We have used the classical Morrey inequality  (see Theorem 7.17 in \cite{GT}) and the fact that
$z^A\geq c|y|^{|A|}$ in $B_{y/2}(x',y)$.

Adding up a geometric series, this yields
\[|u(x',y)-u(x',0)|\leq C\left(\int_{\R^n_*}|\nabla u|^p z^A dz\right)^{1/p} |y|^{1-\frac{D}{p}}\qquad \textrm{for all}\ 0\leq y\leq x_{n-1}.\]
In particular,
\[|u(x',x_n)-u(x',0)|\leq C\left(\int_{\R^n_*}|\nabla u|^p z^A dz\right)^{1/p} |x_n|^{1-\frac{D}{p}}.\]
From this (applied to $x=(x',x_n)$ and to $(x',x_{n-1})$), we deduce
\[|u(x_1,...,x_{n-1},x_n)-u(x_1,...,x_{n-1},x_{n-1})|\leq C\left(\int_{\R^n_*}|\nabla u|^p z^A dz\right)^{1/p} |x|^{1-\frac{D}{p}}.\]

Now, one can repeat the same argument to get
\[|u(x_1,...,x_{n-2},x_{n-1},x_{n-1})-u(x_1,...,x_{n-2},x_{n-2},x_{n-2})|\leq C\left(\int_{\R^n_*}|\nabla u|^p
z^A \right)^{1/p} |x|^{1-\frac{D}{p}}\]
---by comparing $u(x_1,...,x_{n-2},y,y)$ with $u((x_1,...,x_{n-2},y/2,y/2)$ for all $y\leq x_{n-2}$,
and then with $u(x_1,...,x_{n-2},0,0)$ adding up a geometric series.

Proceeding analogously for the rest of the coordinates and then adding all these inequalities, we find that
\[|u(x_1,...,x_1)-u(x)|\leq C\left(\int_{\R^n_*}|\nabla u|^p z^A dz\right)^{1/p} |x|^{1-\frac{D}{p}},\]
and therefore it only remains to control $|u(x_1,...,x_1)-u(0,...,0)|$.

For this, we use a similar argument as in \eqref{easycase}-\eqref{andp}. We first show that there exists $\lambda >1$
depending only on $n$ such that, for any $y>0$,
\begin{equation*}\label{first}
\left|u(y,...,y)-u\left(y/\lambda,...,y/\lambda\right)\right|\leq
C\left(\int_{\R^n_*}|\nabla u|^p z^A dz\right)^{1/p} |y|^{1-\frac{D}{p}}.
\end{equation*}
This is proved applying Theorem 7.17 of \cite{GT} in the ball $B_{\sqrt{n}(1-1/\lambda)y}(y,\ldots,y)$,
whose closure contains the point $(y/\lambda,...,y/\lambda)$. We also
use that $z_i\geq y/2$ for all index $i$ and all $z\in B_{\sqrt{n}(1-1/\lambda)y}(y,\ldots,y)$,
if we take $\lambda>1$ close enough to 1.
Taking $y=x_1$ and adding all these inequalities, we deduce
\[|u(x_1,...,x_1)-u(0)|\leq C\left(\int_{\R^n_*}|\nabla u|^p z^A dz\right)^{1/p} |x|^{1-\frac{D}{p}}.\]

As mentioned above, from this it follows that
\[|u(x)-u(0)|\leq C\left(\int_{\R^n_*}|\nabla u|^p z^A dz\right)^{1/p} |x|^{1-\frac{D}{p}},\]
and hence the lemma is proved.
\end{proof}

We can now give the:

\begin{proof}[Proof of Theorem \ref{morrey}.]
Let us show that
\begin{equation}\label{morr2}
\frac{|u(y)-u(z)|}{|y-z|^{1-\frac Dp}}\leq C\left(\int_{\mathbb R^n_*}x^A|\nabla u|^pdx\right)^{\frac1p}
\end{equation}
for all $y$ and $z$ in $\R^n_*$.
We split the proof of \eqref{morr2} in three steps.

\emph{Step 1.} First, by Lemma \ref{lemamorrey}, we have that \eqref{morr2} holds for $z=0$.

\emph{Step 2.}
We now prove \eqref{morr2} for $y$ and $z$ in $\mathbb R^n_*$ such that $y-z\in\mathbb R^n_*$.
Applying the inequality in Step 1 to the function $v(\tilde y)=u(\tilde y+z)$, $\tilde y\in\R^n$, at the point
$\tilde y=y-z\in \mathbb R^n_*$, we deduce
\[|u(y)-u(z)|\leq C \left(\int_{z+\mathbb R^n_*}(x-z)^A|\nabla u(x)|^pdx\right)^{\frac1p}|y-z|^{1-\frac{D}{p}},\]
where $z+\R^n_*=\{x\in \R^n\,:\, x-z\in\R^n_*\}$.
Therefore, since $(x-z)^A\leq x^A$ if $x$ and $x-z$ belong to $\mathbb R^n_*$, this case of \eqref{morr2} follows.

\emph{Step 3.}
We finally prove \eqref{morr2} for all $y$ and $z$ in $\mathbb R^n_*$.
Define $w\in\R^n_*$ as $w_i=\min\{y_i,z_i\}$ for all $i$.
Then, it is clear that $y-w\in\mathbb R^n_*$ and $z-w\in \mathbb R^n_*$.
Hence, we can apply the inequality proved in Step 2 to obtain
\[|u(y)-u(w)|\leq C \left(\int_{\mathbb R^n_*}x^A|\nabla u|^pdx\right)^{\frac1p}|y-w|^{1-\frac{D}{p}}\]
and
\[|u(z)-u(w)|\leq C \left(\int_{\mathbb R^n_*}x^A|\nabla u|^pdx\right)^{\frac1p}|z-w|^{1-\frac{D}{p}}.\]
Since $|y-w|^2+|z-w|^2=|y-z|^2$, from these two inequalities we deduce that
\[|u(y)-u(z)|\leq 2C \left(\int_{\mathbb R^n_*}x^A|\nabla u|^pdx\right)^{\frac1p}|y-z|^{1-\frac{D}{p}}\]
for all $y,z\in\mathbb R^n_*$.
This finishes the proof of \eqref{morr2}.

Let us prove now \eqref{cormorrey}.
Let $x_0\in\Omega\subset\R^n$ be such that $\sup_\Omega |u|=|u(x_0)|$.
After a finite number of reflections with respect to the coordinate hyperplanes, we may assume that $x_0\in\overline{\R^n_*}$.
Call $\tilde u$ the function $u$ after doing such reflections, defined in the reflected domain $\widetilde \Omega$.
Since $\tilde u\equiv0$ on $\partial\widetilde\Omega$, we have
\[\sup_{\Omega}|u|\cdot{\rm diam}(\Omega)^{-1+\frac{D}{p}}=|\tilde u(x_0)|\cdot{\rm diam}(\widetilde\Omega)^{-1+\frac{D}{p}}\leq \sup_{x,\,y\in\R^n_*}\frac{|\tilde u(x)-\tilde u(y)|}{|x-y|^{1-\frac{D}{p}}}.\]
The right hand side of this inequality is now bounded using \eqref{holder}.
The proof is finished controlling the integral over $\R^n_*$ in \eqref{holder} by an integral over $\Omega\subset\R^n$.
This is needed because of the reflections done initially.
\end{proof}

\section{Weighted Trudinger inequality and proof of Corollary \ref{embeddings}}
\label{sec6}

In this section we prove Theorem \ref{trudinger} and Corollary \ref{embeddings}.
The proof of the weighted Trudinger inequality is based on a bound for the best constant of the weighted Sobolev inequality as $p$ goes to $D$.
Then, the result follows by expanding $\exp(\cdot)$ as a power series and using the weighted Sobolev inequality in each term.
To prove the convergence of this series we need the mentioned bound, which is stated in the following result.

\begin{lem}\label{estimacioCp}
Let $A$ be a nonnegative vector in $\R^n$, $D=A_1+\cdots+A_n+n$, and $p$ be such that $1<p<D$.
Let $C_p$ be the optimal constant of the Sobolev inequality \eqref{sobolev}, given by \eqref{bestcnt}-\eqref{bestcnt2}.
Then,
\[C_p\leq C_0p_*^{1-\frac1D},\]
where $p_*=\frac{pD}{D-p}$ and $C_0$ is a constant which depends only on $D$.
\end{lem}

\begin{proof} The optimal constant is given by
\[C_p=C_1D^{1-\frac{1}{D}-\frac1p}\left(\frac{p-1}{D-p}\right)^{\frac{1}{p'}}\left(\frac{p'\Gamma(D)}{\Gamma\left(\frac{D}{p}\right)\Gamma\left(\frac{D}{p'}\right)}\right)^{\frac1D},\]
where $p'=p/(p-1)$ and $C_1$ is a constant which only depends on $A$ and $n$.
It is easy to see that the constant $C_p$ is bounded as $p\downarrow1$.
Thus, we only have to look at the limit $p\uparrow D$.
It follows from the above expression that
\[C_p\leq C(D-p)^{-\frac{1}{p'}},\]
where $C$ does not depend on $p$.
Hence, taking into account that $\frac{1}{p'}=1-\frac1D-\frac{1}{p_*}$ and $D-p=pD/p_*$, we deduce
\[C_p\leq C_0p_*^{1-\frac1D-\frac1{p_*}}\leq C_0p_*^{1-\frac1D}.\]
Finally, it is easy to see that $C_1$ --- which is given by \eqref{bestcnt} --- can be bounded by a constant depending only on $D$, and therefore we can choose the constant $C_0$ to depend only on $D$.
\end{proof}

We can now give the:

\begin{proof}[Proof of Theorem \ref{trudinger}.] Let $u\in C^1_c(\Omega)$.
From Theorem \ref{sob} and Lemma \ref{estimacioCp} we deduce that
\[\int_\Omega x^A|u|^qdx\leq C_0^q q^{q-\frac{q}{D}}\left(\int_\Omega x^A|\nabla u|^{\frac{qD}{q+D}}dx\right)^{\frac{q+D}{D}}\]
for each $q>1$, where $C_0$ is a constant which depends only on $D$.
Moreover, by H\"older's inequality,
\[\int_\Omega x^A|\nabla u|^{\frac{qD}{q+D}}dx\leq \left(\int_\Omega x^Adx\right)^{\frac{D}{q+D}}\left(\int_\Omega x^A|\nabla u|^D dx\right)^{\frac{q}{q+D}},\]
and thus
\begin{equation}\label{cosatrud}
\int_\Omega x^A|u|^qdx\leq m(\Omega)C_0^q q^{q\frac{D-1}{D}}\|\nabla u\|_{L^D(\Omega,x^Adx)}^{q}.\end{equation}

Now, dividing the function $u$ by some constant if necessary, we can assume
\[\|\nabla u\|_{L^D(\Omega,x^Adx)}=1.\]
Let $c_1$ be a positive constant to be chosen later.
Then, using \eqref{cosatrud} with $q=\frac{kD}{D-1}$, $k=1,2,3,...$, we obtain
\begin{eqnarray} \int_\Omega \exp\left\{\left(c_1|u|\right)^{\frac{D}{D-1}}\right\}x^Adx \nonumber
&=&m(\Omega)+\sum_{k\geq1}\frac{c_1^{\frac{kD}{D-1}}}{k!} \int_\Omega |u|^{\frac{kD}{D-1}}x^Adx\\ \nonumber
&\leq&m(\Omega)+m(\Omega)\sum_{k\geq1}\frac{c_1^{\frac{kD}{D-1}}}{k!} (C_0)^{\frac{kD}{D-1}} \left(\frac{kD}{D-1}\right)^{k} \\
&= &m(\Omega)+m(\Omega)\sum_{k\geq1}\frac{k^k}{k!} \left(\frac{D}{D-1}(c_1C_0)^{\frac{D}{D-1}}\right)^{k}.\label{series}
\end{eqnarray}
Choose $c_1$ (depending only on $D$) satisfying $\frac{D}{D-1}(c_1C_0)^{\frac{D}{D-1}}<\frac{1}{e}$.
Then, by Stirling's formula
\[k!\sim \left(\frac ke\right)^k\sqrt{2\pi k},\]
we deduce that the series \eqref{series} is convergent, and thus
\[\int_\Omega \exp\left\{\left(\frac{c_1|u|}{\|\nabla u\|_{L^D(\Omega,x^Adx)}}\right)^{\frac{D}{D-1}}\right\}x^Adx \leq C_2m(\Omega),\]
as claimed.
Note that the constants $c_1$ and $C_2$ depend only on $D$.
\end{proof}

To end the paper, we give the

\begin{proof}[Proof of Corollary \ref{embeddings}.]
It follows from Theorems \ref{sob}, \ref{morrey}, and \ref{trudinger}.
For a domain $\Omega\subset\R^n$ that is not contained in $\R^n_*$, these results need to be applied to the intersections of $\Omega$ with each of the $2^k$ quadrants, where $k$ is the number of positive entries of the vector $A$ --- see the proof of \eqref{cormorrey} in Theorem \ref{morrey}.
\end{proof}


\begin{thebibliography}{99}


\bibitem{CSCM} X. Cabr\'e, \emph{Partial differential equations, geometry, and stochastic control} (in Catalan), Butl. Soc. Catalana Mat. 15 (2000), 7-27.

\bibitem{CDCDS} X. Cabr\'e, \emph{Elliptic PDEs in Probability and Geometry. Symmetry and regularity of solutions}, Discrete Contin. Dyn. Syst. 20 (2008), 425-457.

\bibitem{CR} X. Cabr\'e, X. Ros-Oton, \emph{Regularity of minimizers up to dimension 7 in domains of double revolution},
Comm. Partial Differential Equations, 38 (2013), 135-154.

\bibitem{CRS-CRAS} X. Cabr\'e, X. Ros-Oton, J. Serra, \emph{Euclidean balls solve some isoperimetric problems with nonradial weights},
C. R. Math. Acad. Sci. Paris 350 (2012), 945-947.

\bibitem{CRS} X. Cabr\'e, X. Ros-Oton, J. Serra, \emph{Sharp isoperimetric inequalities via the ABP method}, Preprint available at http://arxiv.org/abs/1304.1724.

\bibitem{CRS2} X. Cabr\'e, X. Ros-Oton, J. Serra, \emph{Regularity of solutions to elliptic equations with monomial weights}, in preparation.

\bibitem{CS} L. Caffarelli, L. Silvestre, \emph{An extension problem related to the fractional Laplacian}, Comm. Partial Differential Equations 32 (2007), 1245-1260.

\bibitem{C} A. Cavallucci, \emph{Alcuni teoremi di immersione per spazi con peso}, (in Italian) Ann. Mat. Pura Appl. 82 (1969), 143-172.

\bibitem{Ch} I. Chavel, \emph{Riemannian Geometry: A Modern Introduction}, 2nd Revised Edition, Cambridge University Press, Cambridge, 2006.

\bibitem{E} L. C. Evans, \emph{Partial Differential Equations}, Graduate Studies in Mathematics, 1998.

\bibitem{FKS} E. Fabes, C. Kenig, R. Serapioni, \emph{The local regularity of solutions of degenerate elliptic equations}, Comm. Partial Differential Equations 7 (1982), 77-116.

\bibitem{GT} D. Gilbarg, N. Trudinger, \emph{Elliptic Partial Differential Equations of Second Order}, Springer-Verlag, 1983.

\bibitem{G} A. Grigor'yan, \emph{Isoperimetric inequalities for Riemannian products}, Math. Notes 38 (1985), 849-854.

\bibitem{H} P. Hajlasz, \emph{Sobolev spaces on an arbitrary metric space}, Potential Anal. 5 (1996), 403-415.

\bibitem{IN} S. Ivanov, A. Nazarov, \emph{On weighted Sobolev embedding theorems for functions with symmetries},  St. Petersburg Math. J. 18 (2007), 77-88.

\bibitem{MS} C. Maderna, S. Salsa, \emph{Sharp estimates for solutions to a certain type of singular elliptic boundary value problems in two dimensions},  Applicable Anal. 12 (1981), 307-321.

\bibitem{M} F. Morgan, \emph{Manifolds with density}, Notices Amer. Math. Soc. 52 (2005), 853-858.

\bibitem{M2} F. Morgan, \emph{Isoperimetric estimates in products},  Ann. Global Anal. Geom. 30 (2006), 73-79.

\bibitem{T} G. Talenti, \emph{A Weighted Version of a Rearrangement Inequality}, Ann. Univ. Ferrara 43 (1997), 121-133.

\bibitem{T2} G. Talenti, \emph{Best constant in Sobolev inequality}, Ann. Mat. Pura Appl. 110 (1976), 353-372.

\bibitem{W} A. Weinstein, \emph{Generalized axially symmetric potential theory}, Bull. Amer. Math. Soc. 59 (1953), 20-38.


\end{thebibliography}
\end{document}